\documentclass{amsart}
\usepackage{amsmath}
\usepackage{amssymb}
\usepackage{amsthm}
\usepackage{dsfont}
\usepackage{cite}
\usepackage{lineno}
\usepackage{subfigure}
\usepackage{graphicx}
\usepackage{multirow}
\usepackage{color}
\usepackage{xspace}
\usepackage{pdfsync}
\usepackage{hyperref} 
\usepackage{algorithmicx}
\usepackage{algpseudocode}
\usepackage{algorithm}
\usepackage{mathtools}
\usepackage{enumitem}
\graphicspath{{MeshfreeFigures/}}

\DeclareMathOperator*{\argmax}{argmax}

\newcommand{\bq}{\begin{equation}}
\newcommand{\eq}{\end{equation}}
\newcommand{\R}{\mathbb{R}}

\newcommand{\abs}[1]{\left\vert#1\right\vert}

\newcommand{\G}{\mathcal{G}}

\newcommand{\Sf}{\mathcal{S}}
\newcommand{\MA}{Monge-Amp\`ere\xspace}

\algnewcommand{\LineComment}[1]{\State \(\triangleright\) #1}

\newtheorem{theorem}{Theorem}
\theoremstyle{lemma}
\newtheorem{lemma}[theorem]{Lemma}
\newtheorem{corollary}[theorem]{Corollary}
\newtheorem{definition}[theorem]{Definition}

\newtheorem{remark}[theorem]{Remark}
\theoremstyle{remark}

\makeatletter
\newcommand\appendix@section[1]{%
\refstepcounter{section}%
\orig@section*{Appendix \@Alph\c@section: #1}%
}
\let\orig@section\section
\g@addto@macro\appendix{\let\section\appendix@section}
\makeatother

\begin{document}

\title[Strong Comparison Principle for Monge-Amp\`ere]{A strong comparison principle for the generalized Dirichlet problem for Monge-Amp\`ere}

\author{Brittany Froese Hamfeldt}
\thanks{This work was partially supported by NSF DMS-1751996.}
\address{Department of Mathematical Sciences, New Jersey Institute of Technology, University Heights, Newark, NJ 07102}
\email{bdfroese@njit.edu}

\begin{abstract}
We prove a strong form of the comparison principle for the elliptic \MA equation, with a Dirichlet boundary condition interpreted in the viscosity sense.  This comparison principle is valid when the equation admits a Lipschitz continuous weak solution.  The result is tight, as demonstrated by examples in which the strong comparison principle fails in the absence of Lipschitz continuity.  This form of comparison principle closes a significant gap in the convergence analysis of many existing numerical methods for the \MA equation.  An important corollary is that any consistent, monotone, stable approximation of the Dirichlet problem for the \MA equation will converge to the viscosity solution.
\end{abstract}

\date{\today}    
\maketitle

\section{Introduction}
The elliptic \MA equation
\bq\label{eq:MA}
\begin{cases}
\det(D^2u(x)) = f(x,u(x),\nabla u(x))\\
u \text{ is convex}
\end{cases}
\eq
is a fully nonlinear elliptic partial differential equation (PDE) that is related to important problems in optical design~\cite{guan_MAoptics}, surface evolution~\cite{OsherSethian}, image processing~\cite{Sapiro}, and optimal transport~\cite{Villani}.

Even in the simple setting of Dirichlet boundary conditions, the equation need not have a classical smooth solution.  A powerful framework for interpreting weak solutions is the notion of the viscosity solution~\cite{CIL}.  Like many elliptic equations, under mild assumptions the \MA equation has a comparison principle: if $u$ is a subsolution and $v$ a supersolution then $u-v$ attains its maximum on the boundary of the domain.  This type of comparison principle is a critical component of many existence and uniqueness results.

Because of its importance in applications, the last several years of seen a surge of interest in the numerical solution of the \MA equation.  A key breakthrough in analysing these numerical methods was provided by Barles and Souganidis in~\cite{BSnum}, who provided conditions that guarantee the convergence of an approximation scheme to the viscosity solution of the underlying PDE.  This result has inspired the development of many new numerical methods for fully nonlinear elliptic equations~\cite{FO_MATheory,feng2021narrow,mirebeau2016minimal,oberman2005convergent,WanMA,oberman2013finite,ObermanWS,FOFiltered,benamou2016monotone,feng2017convergent,Hamfeldt_gaussian,Nochetto_MAConverge,LiSalgado}.

However, one of the key conditions required by the Barles and Souganidis convergence framework has received very little attention to date: the proof assumes that the underlying PDE satisfies a strong form of a comparison principle.  Briefly, this involves interpreting the boundary conditions in a weak (non-classical) sense, and involves a very strong requirement that subsolutions always lie below supersolutions. Unfortunately, this form of strong comparison principle has never been established for the \MA equation~\eqref{eq:MA} or any other fully nonlinear PDE.  In fact, there are several settings where this form of comparison principle is known to fail for the \MA equation~\cite{Hamfeldt_gaussian,jensensmears}.  

A few recent numerical methods have circumvented this issue by providing alternative method-specific proofs of convergence~\cite{feng2017convergent,Hamfeldt_gaussian,Nochetto_MAConverge,LiSalgado}.  Many other existing methods come equipped with an incomplete convergence proof that relies on the unproven assumption of a strong comparison principle~\cite{FO_MATheory,feng2021narrow,mirebeau2016minimal,oberman2005convergent,WanMA,oberman2013finite,ObermanWS,FOFiltered,benamou2016monotone}.

In this article, we provide the first proof of a strong comparison principle for the \MA equation, which holds as long as the Dirichlet problem admits a Lipschitz continuous solution.  This result appears to be tight: we give an example where the strong comparison fails in the absence of this level of regularity.  This result closes a gap in the convergence proof for many existing numerical methods for the \MA equation.  This also opens the door to much simpler convergence analysis in the ongoing development of numerical methods.

This paper is organised as follows.  In section~\ref{sec:background}, we provide important background information on the \MA equation and its numerical approximation, which highlights the critical gap created by the lack of any strong comparison principle in the existing literature.  In section~\ref{sec:failure}, we describe two specific examples for which the \MA equation fails to have a strong comparison principle due to insufficient regularity in the problem and solution data.  In section~\ref{sec:proof}, we state and prove our main theorem on the strong comparison principle, together with a corollary that closes a long-outstanding gap in the analysis of numerical methods for \MA.

\section{Background}\label{sec:background}
\subsection{Viscosity solutions}
The \MA equation has the form 
\bq\label{eq:PDE} F(x,u,\nabla u(x), D^2u(x)) = 0, \quad x\in\Omega\subset\R^n.\eq
When $u$ is restricted to be convex, \MA
belongs to the class of second-order fully nonlinear degenerate elliptic partial differential equations.
\begin{definition}[Degenerate elliptic]\label{def:elliptic}
The operator
$F:\Omega\times\R\times\R^n\times\Sf^n\to\R$
is \emph{degenerate elliptic} if 
\[ F(x,u,p,X) \leq F(x,v,p,Y) \]
whenever $u \leq v$ and $X \geq Y$.
\end{definition}

The convexity constraint can also be absorbed into the equation by creating a globally elliptic extension of the PDE operator onto non-convex functions.  Different extensions are possible\cite{WanMA,Hamfeldt_gaussian,IshiiLions}, but are generally equivalent to an expression of the form
\bq\label{eq:MAconvex}
F(x,u,p,X) = 
-{\det}^+(X) + f(x,u,p), x \in \Omega
\eq
where the modified determinant satisfies
\bq\label{eq:detPlus}
{\det}^+(X) = \begin{cases} \det(X), & X \geq 0\\ <0, & \text{otherwise}. \end{cases}
\eq

In general, fully nonlinear elliptic equations such as the \MA equation need not have classical solutions.  A very powerful approach to interpreting weak solutions for this class of equations is the viscosity solution~\cite{CIL}.  This notion of weak solution tests whether upper (lower) semi-continuous functions are subsolutions (supersolutions) of the PDE via a maximum principle argument that moves derivatives onto smooth test functions.

\begin{definition}[Upper and Lower Semi-Continuous Envelopes]\label{def:envelope}
The \emph{upper and lower semi-continuous envelopes} of a function $u(x)$ are defined, respectively, by
\[ u^*(x) = \limsup_{y\to x}u(y), \]
\[ u_*(x) = \liminf_{y\to x}u(y). \]
\end{definition}

\begin{definition}[Viscosity subsolution (supersolution)]\label{def:subsuper}
An upper (lower) semi-continuous function $u$ is a \emph{viscosity subsolution (supersolution)} of~\eqref{eq:PDE} if for every $\phi\in C^2({\Omega})$, whenever $u-\phi$ has a local maximum (minimum)  at $x \in {\Omega}$, then
\[ 
F_*^{(*)}(x,u(x),\nabla\phi(x),D^2\phi(x)) \leq (\geq)  0 .
\]
\end{definition}
\begin{definition}[Viscosity solution]\label{def:viscosity}
A function $u$ is a \emph{viscosity solution} of~\eqref{eq:PDE} if $u^*$ is a subsolution and $u_*$ a supersolution.
\end{definition}

Many existence and regularity results are now available for viscosity solutions of the Dirichlet problem for the \MA equation~\cite{BardiMannucci,CafNirSpruck,Lions_Remarks,TrudingerWang_MAReview}.  In this work, we will be particularly interested in Lipschitz continuous viscosity solutions, which can be guaranteed under mild assumptions on the problem data.

\begin{theorem}[Existence of Lipschitz solutions~{\cite[Theorem~1.1]{Yazhe_MA}}]\label{thm:lipExist}
Suppose that $\Omega\subset\R^n$ is uniformly convex with $\partial\Omega\in C^{3,1}$ and that there exists an extension $\tilde{g}\in C^{1,1}(\Omega^\epsilon)$ of the Dirichlet data $g$ onto a neighbourhood $\Omega^\epsilon$ of $\partial\Omega$.  Suppose further that $f^{1/n}\in C^1(\bar{\Omega}\times\R\times\R^n)$ is non-negative and non-decreasing in its second argument and satisfies the bound
\bq\label{eq:gradbounds}f(x,\tilde{g}(x),p) \leq \mu \text{dist}(x,\partial\Omega)^\beta (1+\abs{p}^2)^{\alpha/2}, \quad x\in\Omega^\epsilon, p \in \R^n\eq
for some constants $\mu,\alpha \geq 0$ and $\beta \geq \max\{\alpha-n-1,0\}$.  Finally, we suppose that there exists functions $h\in C^1(\Omega), r \in L^1_{loc}(\R^n)$ and constants $N, \lambda$ such that
\bq\label{eq:fBounds}
f(x,N,p) \leq h(x)/r(p), \quad x \in \Omega, p \in \R^n
\eq
where
\bq\label{eq:exist}
\int_\Omega h(x)\,dx < \int_{\R^n} r(p)\,dp, \quad r^{-1}(p) \geq \lambda>0 \text{ for every }p \in \R^n.
\eq
Then the Dirichlet problem for the \MA equation~\eqref{eq:MA} has a unique convex viscosity solution $u\in C^{0,1}(\bar{\Omega})$.
\end{theorem}

\begin{remark}\label{rem:lipschitz}
There are several variations of Theorem~\ref{thm:lipExist} that also guarantee the existence of a Lipschitz continuous viscosity solution; see for example~\cite[Lemma~3.1]{Urbas_MA}, ~\cite[Theorem~1.1]{Martino}, and~\cite[Theorem~2]{Caff_regMA}. 
\end{remark}

\begin{remark}
We remark that the condition~\eqref{eq:exist} represents the usual compatibility condition required for existence in the special case that $f(x,u,p) = h(x)/r(p)$.  This holds automatically when the \MA equation has no dependence on $\nabla u$.  The restriction~\eqref{eq:gradbounds} places some restrictions on the interplay between the strength of the gradient terms and the rate of decay of the right-hand side near the boundary.
\end{remark}

A particularly nice property of many elliptic equations, which immediately yields uniqueness, is a comparison principle.  The \MA equation also satisfies a comparison principle under very general assumptions.

\begin{theorem}[Comparison principle~{\cite[Theorem~V.2]{IshiiLions}}]\label{thm:comparison}
Suppose that $\Omega\subset\R^n$ is open and uniformly convex. Suppose further that $f\in C^0(\bar{\Omega}\times\R\times\R^n)$ is non-negative, non-decreasing in its second argument, and that for every $R>0$ there exists a constant $C_R  \geq 0$ such that
\[ \abs{f^{1/n}(x,u,p)-f^{1/n}(x,u,q)} \leq C_R\abs{p-q} \]
for every $x\in\bar{\Omega}$, $\abs{u} \leq R$, and $\abs{p},\abs{q} \leq R$.  Let $u$ be any bounded upper semi-continuous viscosity subsolution of~\eqref{eq:MAconvex} and $v$ be any bounded lower semi-continuous viscosity supersolution of~\eqref{eq:MAconvex}.  Then
\bq\label{eq:comparison}
\max\limits_{x\in\bar{\Omega}}\{u(x)-v(x)\} \leq \max\limits_{x\in\partial\Omega}\{u(x)-v(x)\}^+.
\eq
\end{theorem}

\subsection{Numerical approximation of elliptic equations}
In this work, we are particularly motivated by the design and analysis of numerical methods for solving fully nonlinear elliptic equations such as the \MA equation. A key breakthrough in this area was provided by Barles and Souganidis~\cite{BSnum}, who demonstrated that a consistent, monotone, and stable approximation scheme will converge to the viscosity solution of the underlying PDE \emph{provided the equation satisfies a strong form of the comparison principle}.

The analysis provided by~\cite{BSnum} and extended by~\cite{oberman2006convergent} applies to finite difference approximations of the form $F^h(x,u(x),u(x)-u(\cdot))$, where $h$ is a small parameter typically related to the resolution of an underlying mesh.

\begin{definition}[Consistency]\label{def:consistency}
The scheme $F^h(x,u(x),u(x)-u(\cdot))$ is \emph{consistent} with the PDE
\[ F(x,u(x),\nabla u(x), D^2u(x)) = 0, \quad x\in\bar{\Omega} \]
 if for any smooth function $\phi$ and $x\in\bar{\Omega}$,
\[ \limsup_{h\to0,y\to x, z\in\G^h\to x,\xi\to0} F^h(z,\phi(y)+\xi,\phi(y)-\phi(\cdot)) \leq F^*(x,\phi(x),\nabla\phi(x),D^2\phi(x)), 
\]
\[ \liminf_{h\to0,y\to x, z\in\G^h\to x,\xi\to0} F^h(y,\phi(y)+\xi,\phi(y)-\phi(\cdot)) \geq F_*(x,\phi(x),\nabla\phi(x),D^2\phi(x)). \]
\end{definition}

\begin{definition}[Monotonicity]\label{def:monotonicity}
The scheme~$F^h(x,u(x),u(x)-u(\cdot))$ is \emph{monotone} if $F^h$ is a non-decreasing function of its final two arguments.
\end{definition}

\begin{definition}[Stability]\label{def:stability}
The scheme~$F^h(x,u(x),u(x)-u(\cdot))$ is \emph{stable} if there exists a constant $M\in\R$, independent of $h$, such that if $u^h$ is any solution to $F^h(x,u^h(x),u^h(x)-u^h(\cdot)) = 0$ than $\|u^h\|_\infty \leq M$ for all sufficiently small $h>0$.
\end{definition}

Critical to the analysis of~\cite{BSnum} is that not only the PDE, but also the boundary conditions, be interpreted in the viscosity sense.  To accomplish this, the PDE operator needs to be defined on the closure of the domain.  For the Dirichlet problem for the \MA equation, this leads to the operator
\bq\label{eq:dirichletMA}
F(x,u,p,X) = \begin{cases} -{\det}^+(X) + f(x,u,p), & x \in \Omega\\
u - g(x),  & x \in\partial\Omega. \end{cases}
\eq

Viscosity solutions of the generalised Dirichlet problem for the \MA equation are then defined by applying Definitions~\ref{def:subsuper}-\ref{def:viscosity} at all points $x\in\bar{\Omega}$~\cite{Urbas_MA}.
 Given a continuous right-hand side $f$, this requires us to consider the following envelopes of the PDE operator~\eqref{eq:dirichletMA}:
\bq\label{eq:MAlower}
F_*(x,u,p,X) = \begin{cases}
-{\det}^+(X) + f(x,u,p), & x\in\Omega\\
\min\{u-g(x),-{\det}^+(X) + f(x,u,p)\}, & x\in\partial\Omega, 
\end{cases}
\eq
\bq\label{eq:MAupper}
F^*(x,u,p,X) = \begin{cases}
-{\det}^+(X) + f(x,u,p), & x\in\Omega\\
\max\{u-g(x),-{\det}^+(X) + f(x,u,p)\}, & x\in\partial\Omega.
\end{cases}
\eq

The difficulty in using the Barles-Souganidis convergence framework in practice is that it requires the PDE, posed on $\bar{\Omega}$, to satisfy a strong form of a comparison principle.
\begin{definition}[Strong comparison]\label{def:comparisonStrong}
A PDE operator $F:\bar{\Omega}\times\R\times\R^n\times\Sf^n\to\R$ satisfies a \emph{strong comparison principle} if, whenever $u$ is a viscosity subsolution and $v$ a viscosity supersolution, then $u \leq v$ on $\bar{\Omega}$.
\end{definition}
This differs from the usual comparison principle (Theorem~\ref{thm:comparison}) in that there is no explicit reference to values of $u$ and $v$ on the boundary; this is all interpreted in a weak sense via the operators~\eqref{eq:MAlower}-\eqref{eq:MAupper}.

In fact, it is not currently known if \emph{any} fully nonlinear PDE satisfies this strong form of the comparison principle.  The Barles-Souganidis framework has inspired the development of many numerical methods for fully nonlinear elliptic equations in recent years. However, many of these have an incomplete convergence proof that relies on the assumption of a strong comparison principle~\cite{FO_MATheory,feng2021narrow,mirebeau2016minimal,oberman2005convergent,WanMA,oberman2013finite,ObermanWS,FOFiltered,benamou2016monotone}.  The proof of a strong comparison principle would close the gaps in many existing convergence proofs and pave the way for the development of new provably convergent methods.

\section{Failure of strong comparison principle}\label{sec:failure}
A strong comparison principle for the \MA equation (or any other nonlinear elliptic PDE) is certainly not trivial.  Indeed, there are settings where the equation is known to \emph{not} possess a strong comparison principle.

\subsection{Non-uniform ellipticity of the domain.}  The first example we discuss was first described in~\cite{jensensmears}.  In this example, strong comparison fails due to non-uniform ellipticity of the domain.

Let $\Omega\subset\R^2$ be the half-plane $\Omega = \{(x_1,x_2)\in\R^2 \mid x_1>0\}$.  We consider the \MA equation with $f(x,u,p) = 0$ and homogeneous Dirichlet data $g(x) = 0$.  The PDE operator is
\bq\label{eq:ex1} F(x,D^2u(x)) = \begin{cases} -{\det}^+(D^2u(x)), & x \in \Omega\\ u(x),  & x \in \partial\Omega.\end{cases} \eq

We propose the following sub(super)solutions of the generalised Dirichlet problem:
\[ u(x) = 0, \quad v(x) = \begin{cases} 0, & x \in \Omega\\ -1, & x \in \partial\Omega.\end{cases} \]
Trivially, $u$ is a solution (and therefore subsolution) of~\eqref{eq:ex1}.

It is also clear that $v$ satisfies the definition of a viscosity supersolution at interior points $x_0\in\Omega$.  We now consider any $x_0\in\partial\Omega$ and verify that $v$ satisfies the definition of a viscosity supersolution of the generalised Dirichlet problem at this point.  Consider any $\phi\in C^2(\bar{\Omega})$ such that $v-\phi$ has a local minimum at $x_0$.  Then we must have
\[ v(x_0) - \phi(x_0) \leq v(x) - \phi(x) \]
for all $x$ near $x_0$.  In particular, consider points of the form $x = x_0 \pm \epsilon y\in\partial\Omega$ where $\epsilon>0$ and $y = (0,1)$ is tangent to $\partial\Omega$ at $x_0$.  Then for sufficiently small $\epsilon$,
\[ \phi(x_0 \pm \epsilon y) - \phi(x_0) \leq v(x_0 \pm \epsilon y) - v(x_0) = 0. \]
Since $\phi\in C^2$, we can combine these results and take $\epsilon\to0$ in the standard centered difference formula to verify that
\[ \frac{\partial^2\phi}{\partial y^2}(x_0) = \lim\limits_{\epsilon\to0}\frac{\phi(x_0+\epsilon)+\phi(x_0-\epsilon)-2\phi(x_0)}{\epsilon^2} \leq 0. \]
From the definition of the extended determinant~\eqref{eq:detPlus}, we conclude that
\[ -{\det}^+(D^2\phi(x_0)) \geq 0, \]
from which we deduce that
\[ F^*(x_0,v(x_0),\nabla\phi(x_0),D^2\phi(x_0)) = \max\{-{\det}^+(D^2\phi(x_0)),v(x_0)\} \geq 0. \]
Therefore $v$ is a viscosity supersolution.

However, it is clear that for $x_0\in\partial\Omega$, $u(x_0) = 0 > -1 = v(x_0)$.  Thus the strong comparison principle fails in this setting.

\subsection{Gradient blow-up.}
The second example we consider was first described in~\cite{Hamfeldt_gaussian}.  In this case, the strong comparison principle fails because of the presence of strong gradient terms in the equation, which in turn allow for the existence of a solution whose gradient blows up at the boundary.

We consider the prescribed Gaussian curvature equation with unit curvature posed on the one-dimensional domain $\Omega = (0,1)$.  The PDE operator is
\bq\label{eq:ex2}
F(x,u,u_x,u_xx) = 
\begin{cases}
-u_{xx} + (1+u_x^2)^{3/2}, & x \in \Omega\\
u+1, & x = 0\\
u-1, & x = 1.
\end{cases}
\eq

Given that the solution to this equation should have constant unit Gaussian curvature, we expect the solution to define a portion of a unit circle.  However, it is not possible to fit a unit circle to the given Dirichlet data ($u(0) = -1$, $u(1) = 1$).  Indeed, the viscosity solution to the generalised Dirichlet problem is known to be discontinuous~\cite{Bakelman,Hamfeldt_gaussian}.

We propose the following sub(super) solutions:
\[
u(x) = \begin{cases}
-\sqrt{1-x^2}, & x\in[0,1)\\ 1, & x = 1,
\end{cases} \quad v(x) = -\sqrt{1-x^2}.\]

It is easy to verify that $u$ is upper semi-continuous and satisfies the definition of a subsolution at all points in $[0,1)$.  Now we verify that $u$ satisfies the definition of a subsolution at $x_0 = 1$.  Consider any $\phi\in C^2$ such that $u-\phi$ has a local max at $x_0$.  Since $u(1) = 1$, we can easily confirm that
\[ F_*(1,u(1),\phi_x(1),\phi_{xx}(1)) = \min\{-\phi_{xx}(1) + (1+\phi_x(1)^2)^{3/2}, u(1) - 1\} \leq 0. \]
Thus $u$ is a viscosity subsolution of the generalised Dirichlet problem.

Next we observe that $v$ is trivially a supersolution of the equation at every $x_0\in[0,1)$ since $v$ is a classical solution of the ODE and $v(0) = -1$.  Now we verify that $v$ is a supersolution at the point $x_0 = 1$.  This requires us to test every smooth $\phi\in C^2$ such that $v-\phi$ has a local max at $x_0 = 1$.  However, we notice that $v$ has an empty subgradient at $x_0 = 1$ since $v_x \to \infty$ as $x \to 1$. Thus there are actually \emph{no} smooth test functions $\phi$ with the required property and the definition of a supersolution is trivially satisfied at $x_0 = 1$.  We conclude that $v$ is a viscosity supersolution of the generalised Dirichlet problem.

However, we notice that $u(1) = 1 > 0 = v(1)$ so the strong comparison principle fails.

\section{Proof of strong comparison principle}\label{sec:proof}
We now prove that the Dirichlet problem for the \MA equation~\eqref{eq:dirichletMA} does have a strong comparison principle in the setting where Lipschitz continuous viscosity solutions exist.  As discussed in Remark~\ref{rem:lipschitz}, there are a variety of different settings that guarantee this level of regularity. The only additional hypotheses required are the usual assumptions needed to ensure the equation has a traditional comparison principle (Theorem~\ref{thm:comparison}).

\begin{theorem}[Strong comparison principle]\label{thm:strongcomparison}
Consider the generalised Dirichlet problem for the \MA equation~\eqref{eq:dirichletMA} where the domain $\Omega$ and right-hand side $f$ satisfy the assumptions of Theorem~\ref{thm:comparison}.  Suppose also that there exists a viscosity solution $u \in C^{0,1}(\bar{\Omega})$ that satisfies $u(x) = g(x)$ on $\partial\Omega$.  Let $v$ be any viscosity subsolution and $w$ any viscosity supersolution of the   generalised Dirichlet problem.  Then $v \leq w$ on $\bar{\Omega}$.
\end{theorem}

Combining this with the framework proposed by Barles and Souganidis~\cite{BSnum} and further developed by Oberman~\cite{oberman2006convergent}, this leads immediately to a very general result on the convergence of monotone finite difference schemes for the \MA equation.

\begin{corollary}[Convergence of monotone approximation schemes]\label{cor:convergence}
Let $u\in C^{0,1}(\bar{\Omega})$ be the unique viscosity solution of the Dirichlet problem for the \MA equation under the hypotheses of Theorem~\ref{thm:strongcomparison}.  Let $F^h$ be any consistent, monotone, stable approximation scheme and let $u^h$ be any solution of the scheme $F^h(x,u^h(x),u^h(x)-u^h(\cdot)) = 0$.  Then $u^h$ converges to $u$ pointwise on $\bar{\Omega}$ as $h\to 0$.
\end{corollary}

This result completes the proof of convergence for many existing numerical methods including~\cite{FO_MATheory,feng2021narrow,mirebeau2016minimal,oberman2005convergent,WanMA,oberman2013finite,ObermanWS,FOFiltered,benamou2016monotone}.

\subsection{Behaviour of sub- and supersolutions}
In order to prove Theorem~\ref{thm:strongcomparison}, we will need to understand the behaviour of viscosity sub- and supersolutions at the boundary of the domain.  Unlike in the classical setting, this is not as simple as asserting that $v \leq g \leq w$ on $\partial\Omega$. In this subsection, we will generalise and tighten several observations made by the author for the equation of prescribed Gaussian curvature~\cite{Hamfeldt_gaussian}.

We begin by noting that viscosity subsolutions are automatically convex, which is a fairly straightforward consequence of~\cite[Theorem~1]{ObermanCE}.
\begin{lemma}[Subsolutions are convex]\label{lem:subConvex}
Let $v$ be an upper semi-continuous sub-solution of the \MA equation~\eqref{eq:dirichletMA}.  Then $v$ is convex.
\end{lemma}
\begin{proof}
Choose $x_0\in\Omega$ and $\phi\in C^2$ such that $v-\phi$ has a local maximum at $x_0$.  Since $v$ is a subsolution of~\eqref{eq:dirichletMA},
\[ -{\det}^+(D^2\phi(x_0)) + f(x,v(x_0),\nabla\phi(x_0)) \leq 0. \]
An immediate consequence of this is that ${\det}^+(D^2\phi(x_0)) \geq 0$.  From the definition of the extended determinant operator~\eqref{eq:detPlus}, this is only possible if $D^2\phi(x_0) \geq 0$.  This, in turn, requires that the smallest eigenvalue of the Hessian $\lambda_1(D^2\phi(x_0))$ is non-negative
 and therefore $v$ is also a sub-solution of
\[ -\lambda_1(D^2v(x)) = 0. \]
This is precisely the hypothesis of~\cite[Theorem~1]{ObermanCE}, which establishes the convexity of $v$. 
\end{proof}

Next, we observe that subsolutions of the generalised Dirichlet problem are actually subsolutions of the Dirichlet boundary condition in the usual sense.
\begin{lemma}[Behaviour of subsolutions at boundary]\label{lem:subBC}
Let $v$ be an upper semi-continuous sub-solution of~\eqref{eq:dirichletMA}.  Then $v\leq g$ on $\partial\Omega$.
\end{lemma}

\begin{proof}
Choose any $x_0\in\partial\Omega$ and small $\epsilon>0$.  Since $\Omega$ is convex, there exists a supporting hyperplane to the domain at $x_0$.  We  let $n(x_0)$ be the unit outward normal to any such hyperplane.  Since $\Omega$ is uniformly convex, there exists some $\alpha>0$ such that for any $x\in\bar{\Omega}$ with $\abs{x-x_0}$ sufficiently small,
\[ -n(x_0)\cdot(x-x_0) \geq \alpha\abs{x-x_0}^2. \]

Denote by $B_\epsilon$ the open ball $B(x_0,\epsilon)$.  
We let $\gamma_\epsilon = \dfrac{1}{\alpha\epsilon^2}\left(\max\limits_{\partial B_\epsilon\cap\bar{\Omega}}v-v(x_0)+\epsilon\right)$ and define the hyperplane
\[ P_\epsilon(x) \equiv v(x_0)-\gamma_\epsilon n(x_0)\cdot(x-x_0). \]
Then we notice that for any $x\in\partial B_\epsilon\cap\bar{\Omega}$,
\bq\label{eq:hyperplaneBdy} P_\epsilon(x) \geq v(x_0) + \gamma_\epsilon\alpha\epsilon^2>\max\limits_{\partial B_\epsilon\cap\bar{\Omega}}v \geq v(x). \eq

Since $v$ is upper semi-continuous, there exists some
\[ z_\epsilon \in \argmax\limits_{\bar{B_\epsilon}\cap\bar{\Omega}}\{v-P_\epsilon\}. \]
We note that $v-P_\epsilon<0$ on $\partial B_\epsilon\cap\bar{\Omega}$ by~\eqref{eq:hyperplaneBdy} and $v(x_0)-P_\epsilon(x_0) = 0$.  Thus the maximiser $z_\epsilon\notin\partial B_\epsilon\cap\bar{\Omega}$ and the local maximum satisfies $v(z_\epsilon)-P_\epsilon(z_\epsilon) \geq 0$.

Consider any $x\in B_\epsilon\cap\Omega$.  As the intersection of two convex sets, $B_\epsilon\cap\Omega$ is also convex.  Since $x$ is in the interior of this convex set, it can be expressed as $\lambda_1 x_1 + \lambda_2 x_2$ for some $x_1\in\partial B_\epsilon\cap\Omega$, $ x_2 \in B_\epsilon\cap\Omega$ and $\lambda_1, \lambda_2 > 0$ with $\lambda_1+\lambda_2 = 1$.  We have $v(x_1)-P_\epsilon(x_1)<0$ since $x_1$ is on the boundary of the ball~\eqref{eq:hyperplaneBdy}.  Then using the fact that $v$ is convex (Lemma~\ref{lem:subConvex}) and $P_\epsilon$ is affine, we can calculate
\begin{align*}
v(x)-P_\epsilon(x) &= v(\lambda_1 x_1 + \lambda_2 x_2) - P_\epsilon(\lambda_1 x_1 + \lambda_2 x_2)\\
 &\leq \lambda_1(v(x_1)-P_\epsilon(x_1)) + \lambda_2(v(x_2)-P_\epsilon(x_2))\\
&< \lambda_2(v(x_2)-P_\epsilon(x_2))\\
 &\leq v(z_\epsilon)-P_\epsilon(z_\epsilon).
\end{align*}
Therefore the maximiser cannot be in the interior of $B_\epsilon\cap\Omega$.  The only remaining possibility is $z_\epsilon\in B_\epsilon\cap\partial\Omega$.

As $\Omega$ is uniformly convex, there exists $\beta>0$ such that whenever $x\in\bar{\Omega}$,
\[ (x-z_\epsilon)\cdot n(z_\epsilon) \leq -\beta\abs{x-z_\epsilon}^2. \]

Define the test function
\[ \phi_\epsilon(x) = P_\epsilon(x)-(x-z_\epsilon)\cdot n(z_\epsilon) - \beta\abs{x-z_\epsilon}^2 \in C^2. \]
We notice that
\[ v(x)-\phi_\epsilon(x) \leq v(x)-P_\epsilon(x) \leq v(z_\epsilon)-P_\epsilon(z_\epsilon) = v(z_\epsilon)-\phi(z_\epsilon) \]
for $x\in\bar{\Omega}$ sufficiently close to $z_\epsilon$.
Thus $v-\phi_\epsilon$ has a local maximum at $z_\epsilon$.  Since $v$ is a subsolution, this requires
\bq\label{eq:vsub}\begin{split} F_*(z_\epsilon,v(z_\epsilon)&,\nabla\phi_\epsilon(z_\epsilon),D^2\phi_\epsilon(z_\epsilon))\\ &= \min\{-{\det}^+(D^2\phi_\epsilon(z_\epsilon)) + f(z_\epsilon,v(z_\epsilon),\nabla\phi_\epsilon(z_\epsilon)),v(z_\epsilon)-g(z_\epsilon)\} \leq 0. \end{split}\eq

However, by construction, $D^2\phi_\epsilon(z_\epsilon) = -2\beta I < 0$, which means that ${\det}^+(D^2\phi_\epsilon(z_\epsilon)) < 0$ as well.
In order to satisfy~\eqref{eq:vsub}, we must have
\[ v(z_\epsilon)-g(z_\epsilon) \leq 0.\]

Now we return to the observation that $z_\epsilon$ is a maximiser of $v-P_\epsilon$ near $x_0$.  In particular, this means that
\begin{align*}
v(x_0) &\leq v(z_\epsilon)-P_\epsilon(z_\epsilon) + P_\epsilon(x_0) \\
  &\leq g(z_\epsilon)-P_\epsilon(z_\epsilon) + P_\epsilon(x_0)\\
	&= g(z_\epsilon) + \gamma_\epsilon n(x_0)\cdot(z_\epsilon-x_0)\\
	&\leq g(z_\epsilon) - \gamma_\epsilon \alpha \abs{z_\epsilon-x_0}^2\\
	&\leq g(z_\epsilon).
\end{align*}

Since the boundary data $g$ is continuous, we can take $\epsilon\to0$ to obtain
\[ v(x_0) \leq g(x_0). \qedhere \]
\end{proof}

Finally, we observe that supersolutions of the generalised Dirichlet problem do \emph{not} need to be supersolutions in the usual sense (i.e it is not necessary for $w \geq g$ at the boundary).  However, this condition can only be violated at points where the subgradient of $w$ is empty.

Here we use the usual definition of the subgradient of a function.  While typically this concept is used in the context of convex functions, we will allow this same definition apply to more general non-convex functions as well.  This, of course, prevents us from utilising any of the usual results regarding the subgradient of a convex function.
\begin{definition}[Subgradient]\label{def:subgradient}
The \emph{subgradient} of a function $u$ at a point $x_0\in\bar{\Omega}$ is the set
\[ \partial u(x_0) = \{p \mid u(x) \geq u(x_0) + p \cdot(x-x_0) \text{ for every } x \in \bar{\Omega}\}. \]
\end{definition}

\begin{lemma}[Behaviour of supersolutions at boundary]\label{lem:superBC}
Let $w$ be a lower semi-continuous super-solution of~\eqref{eq:dirichletMA}.  Then at each $x_0\in\partial\Omega$, either $w(x_0) \geq g(x_0)$ or the subgradient $\partial w(x_0)$ is empty.
\end{lemma}

\begin{proof}
Let $x_0\in\partial\Omega$ and suppose that there exists some $p\in\partial w(x_0)$.  Consider any supporting hyperplane to the domain at $x_0$ and let $n$ be the unit outward normal to this hyperplane.  
Since $\Omega$ is uniformly convex, there exists some constant $\alpha>0$ such that for small enough $\abs{x-x_0}$ with $x\in\bar{\Omega}$,
\[ (x-x_0)\cdot n \leq -\alpha\abs{x-x_0}^2. \]

Now we choose any $\gamma > 0$ and consider the test function
\[ \phi(x) = w(x_0) + p\cdot(x-x_0) +  (x-x_0)\cdot n + \frac{\alpha}{2}\abs{x-x_0}^2 + \frac{\gamma}{2}\left((x-x_0)\cdot n\right)^2.  \]
By the definition of $p$, we have
\[ w(x_0) + p\cdot(x-x_0) \leq w(x). \]
From the definition of $\alpha$ we know that
\[ \frac{1}{2} (x-x_0)\cdot n + \frac{\alpha}{2}\abs{x-x_0}^2 \leq 0. \]
Finally, as long as $\abs{x-x_0} < 1/\gamma$ we have
\[  \frac{1}{2} (x-x_0)\cdot n + \frac{\gamma}{2}\left((x-x_0)\cdot n\right)^2 \leq 0.\]
Putting these results together, we obtain
\[ \phi(x) \leq w(x) \]
near $x_0$, with $\phi(x_0) = w(x_0)$.  Thus $w-\phi$ has a local minimum at $x_0$.

We also note that $\phi\in C^2$ and
\begin{align*}
\nabla\phi(x_0) &= p + n,\\
D^2\phi(x_0) &= \alpha I + \gamma nn^T > 0.
\end{align*}

Then for sufficiently large $\gamma$:
\bq\label{eq:wTest}
\begin{split}
-{\det}^+(D^2\phi(x_0)) &+ f(x_0,w(x_0),\nabla\phi(x_0))\\ &= -\det(\alpha I + \gamma nn^T) + f(x_0,w(x_0),p+n)<0.
\end{split}\eq
However, since $w$ is a supersolution, we know that
\bq\label{eq:wsuper}\begin{split} F^*(x_0,w(x_0)&,\nabla\phi(x_0),D^2\phi(x_0))\\ &= \max\{-{\det}^+(D^2\phi(x_0)) + f(x_0,w(x_0),\nabla\phi(x_0)),w(x_0)-g(x_0)\} \geq 0. \end{split}\eq

The only way for both~\eqref{eq:wTest} and~\eqref{eq:wsuper} to hold is if $w(x_0) \geq g(x_0)$.
\end{proof}

\subsection{Proof of main theorem}
We are now ready to complete the proof of our main theorem by combining the observations of the previous subsection with the traditional comparison principle (Theorem~\ref{thm:comparison}).

\begin{proof}[Proof of Theorem~\ref{thm:strongcomparison}]
We recall that $v$ is an upper semi-continuous subsolution and $w$ a lower semi-continuous supersolution of the generalised Dirichlet problem for the \MA equation, while $u$ is a Lipschitz continuous viscosity solution that satisfies the Dirichlet boundary conditions in the usual sense. Then $u$ is both a sub- and supersolution to the generalised Dirichlet problem.

We know from Lemma~\ref{lem:subBC} that $v \leq g$ on $\partial\Omega$.

We now suppose that \bq\label{eq:assume}\sup\limits_{x\in\partial\Omega}\{g(x)-w(x)\}>0\eq and seek a contradiction.  To accomplish this, we use a traditional comparison principle (Theorem~\ref{thm:comparison}) to compare $u$ and $w$.   Since $u-w$ is upper semi-continuous and $u=g$ on $\partial\Omega$, we can find some $x_0\in\partial\Omega$ such that for every $x\in\bar{\Omega}$,
\bq\label{eq:compInProof} \begin{split}u(x)-w(x) &\leq \sup\limits_{x\in\partial\Omega}\{u(x)-w(x)\}^+ = \sup\limits_{x\in\partial\Omega}\{g(x)-w(x)\}^+\\ &=\sup\limits_{x\in\partial\Omega}\{g(x)-w(x)\}= g(x_0)-w(x_0)= u(x_0)-w(x_0). \end{split}\eq

Now since $u\in C^{0,1}(\bar{\Omega})$ is convex, it has a non-empty subgradient at $x_0$. That is, there exists some $p\in\R^n$ such that
\[ u(x) \geq u(x_0) + p \cdot(x-x_0) \]
for every $x\in\bar{\Omega}$.  Combining this with the result of the traditional comparison principle in~\eqref{eq:compInProof}, we find that
\[ w(x) \geq w(x_0) + u(x) - u(x_0) \geq w(x_0) + p\cdot(x-x_0) \]
for every $x\in\bar{\Omega}$. That is, $p\in\partial w(x_0)$.

Since $\partial w(x_0)$ is non-empty and $w$ is a supersolution, we have from Lemma~\ref{lem:superBC} that $w(x_0) \geq g(x_0)$.  Combining this with the definition of $x_0$ in~\eqref{eq:compInProof}, we find that
\[\sup\limits_{x\in\partial\Omega}\{g(x) - w(x)\} = g(x_0)-w(x_0) \leq 0. \]
This contradictions the assumption in~\eqref{eq:assume} and we conclude that actually $w \geq g$ on $\partial\Omega$.

We combine our observations and note that $v \leq g \leq w$ on $\partial\Omega$.  Now we use the traditional comparison principle (Theorem~\ref{thm:comparison}) one more time to compare $v$ and $g$.  This leads to the conclusion that
\[ \sup\limits_{x\in\bar{\Omega}}\{v(x)-w(x)\} \leq \sup\limits_{x\in\partial\Omega}\{v(x)-w(x)\}^+ = 0 \]
so that $v \leq w$ on $\bar{\Omega}$.
\end{proof}

\bibliographystyle{plain}
\bibliography{StrongComparison}

\end{document}